\documentclass{article}

\usepackage{amsmath,amssymb}
\usepackage{amsthm}
\usepackage{url}

\usepackage{algorithm,multicol}
\usepackage{algpseudocode}

\newtheorem{theorem}{Theorem}[section]
\newtheorem{corollary}[theorem]{Corollary}
\newtheorem{lemma}[theorem]{Lemma}

\theoremstyle{remark}
\newtheorem{example}{Example}

\begin{document}

\title{On a nonlinear relation for computing the overpartition function}
\author{Mircea Merca\footnote{mircea.merca@profinfo.edu.ro}
 \\ 
\small Department of Mathematics, University of Craiova, 200585 Craiova, Romania\\
\small Academy of Romanian Scientists, Ilfov 3, Sector 5, Bucharest, Romania\\}

\date{}
\maketitle

\begin{abstract}
In 1939, H. S. Zuckerman provided a Hardy-Ramanujan-Rademacher-type convergent series that can be used to compute an isolated value of the overpartition function $\overline{p}(n)$. Computing $\overline{p}(n)$ by this method requires arithmetic with very high-precision approximate real numbers and it is complicated. In this paper, we provide a formula to compute the values of $\overline{p}(n)$ that requires only the values of $\overline{p}(k)$ with $k\leqslant n/2$. This formula is combined with a known linear homogeneous recurrence relation for the overpartition function $\overline{p}(n)$ to obtain a simple and fast computation of the value of $\overline{p}(n)$. This new method uses only (large) integer arithmetic and it is simpler to program.  
\\
\\
{\bf Keywords:} algorithms, partitions, overpartitions, recurrences
\\
\\
{\bf MSC 2010:} 05A17, 05A19, 11P81, 11P82 
\end{abstract}

\section{Introduction}
\label{intro}

Recall \cite{Corteel} that an overpartition of the positive integer $n$ 
is an ordinary partition of $n$ where the first occurrence of parts of each size may be overlined. 
Let $\overline{p}(n)$ denote the number of overpartitions of $n$.
For example, the overpartitions of the integer $3$ are:
$$3,\ \overline{3},\ 2+1,\ \overline{2}+1,\ 2+\overline{1},\ \overline{2}+\overline{1},\ 1+1+1\ \text{and}\ \overline{1}+1+1.$$
We see that $\overline{p}(3)=8$. 
It is well-known that the generating function of $\overline{p}(n)$ is given by
\begin{equation}\label{Eq1}
\sum_{n=0}^{\infty} \overline{p}(n) q^n = \frac{(-q;q)_{\infty}}{(q,q)_{\infty}} = \left(\sum_{n=-\infty}^\infty (-q)^{n^2} \right)^{-1},
\end{equation}
where
$$(a;q)_{\infty} = \lim_{n\to\infty} (1-a)(1-aq)(1-aq^2)\cdots(1-aq^{n-1}).$$
Because the infinite product $(a;q)_{\infty}$ diverges when $a\neq 0$ and $|q|\geqslant 1$, whenever $(a;q)_{\infty}$ appears in a formula, we shall assume that $|q|<1$.

Overpartitions were introduced by Corteel and Lovejoy in \cite{Corteel} 
and have been the subject of many recent studies including 
Andrews \cite{Andrews15}, 
Bringmann and Lovejoy \cite{Bringmann},
Chen and Zhao \cite{Chen},
Corteel and Lovejoy \cite{Corteel}, 
Corteel and Hitczenko \cite{Corteel04}, 
Corteel, Goh and Hitczenko \cite{Corteel06},
Corteel and Mallet \cite{Corteel07},
Fu and Lascoux \cite{Fu},
Hirschhorn and Sellers \cite{Hirschhorn,Hirschhorn06}, 
Kim \cite{Kim}, 
Lovejoy \cite{Lovejoy03,Lovejoy04,Lovejoy04a,Lovejoy05,Lovejoy05a,Lovejoy07}, 
Mahlburg \cite{Mahlburg},
Merca \cite{Merca17a}
and Sills \cite{Sills}.

The following linear homogeneous recurrence relation \cite[Corollary 4]{Fortin}
\begin{equation}\label{Eq:0}
\overline{p}(n)+2 \sum_{j=1}^{\left\lfloor\sqrt{n} \right\rfloor} (-1)^{j} \overline{p}(n-j^2)=0,
\end{equation}
with $\overline{p}(0)=1$ 
provides a simple and reasonably efficient way to compute the value of $\overline{p}(n)$.
In fact, computing the value of $\overline{p}(n)$ with this recurrence relation requires all the values of $\overline{p}(k)$ with $k<n$.

There is a better way to compute an isolated value of $\overline{p}(n)$.
More than $80$ years before the coining of the term \textit{overpartition}, Hardy and Ramanujan \cite[p. 109--110]{Hardy} went on to state that
\begin{equation*}\label{Eq:1}
\overline{p}(n) = \frac{1}{4\pi} \frac{d}{dt} \left( \frac{e^{\pi\sqrt{n}}}{\sqrt{n}} \right)
+\frac{\sqrt{3}}{2\pi} \cos\left(\frac{2n\pi}{3}-\frac{\pi}{6} \right) \frac{d}{dn} \left( e^{\pi\sqrt{n}/3} \right)
+\cdots+O(n^{-1/4}).
\end{equation*}
This result was improved by Zuckerman \cite{Zuckerman} to the following 
Hardy-Ramanujan-Rademacher-type convergent series:
\begin{equation*}
\overline{p}(n) = \frac{1}{2\pi} \sum_{\substack{k=1 \\ 2\nmid k}}^{\infty} \sqrt{k} 
\sum_{\substack{0\leqslant h<k \\ \gcd(h,k)=1}} \frac{\omega(h,k)^2}{\omega(2h,k)} e^{-2\pi i n h/k}
\frac{d}{dn}\left( \frac{\sinh(\pi \sqrt{n}/k)}{\sqrt{n}}\right),
\end{equation*}
where
$$\omega(h,k) = \exp\left( \pi i \sum_{r=1}^{k-1} \frac{r}{k} \left( \frac{hr}{k} - \left\lfloor \frac{hr}{k} \right\rfloor -\frac{1}{2} \right) \right).$$ 
Computing $\overline{p}(n)$ by this formula requires arithmetic with very high-precision approximate real numbers and it is complicated. 
Details on how to efficiently implement a Hardy–Ramanujan–Rademacher type formula can be found in \cite{Johansson}.

In this paper, we present a new recurrence formula for computing the value of $\overline{p}(n)$ that requires only the values of $\overline{p}(k)$ with $k\leqslant n/2$. 
This new recurrence formula is not linear, uses only (large) integer arithmetic, it is simpler to program.

\begin{theorem}\label{T1}
	For $n\geqslant 0$,
	\begin{equation}\label{eqT1}
	\overline{p}(n)- \sum_{k=0}^{\lfloor n/2 \rfloor} \sum_{j=-\infty}^{\infty} \overline{p}(k)\overline{p}\Big(\left\lfloor n/2 \right\rfloor -k-j\big(2j+1-(-1)^n\big)\Big)=0.
	\end{equation}		
\end{theorem}

This identity can be written in a more explicit form in the following way considering that $[x]=\lfloor x+1/2 \rfloor$.

\begin{corollary}\label{C1}
	For $n\geqslant 0$,
	\begin{enumerate}
		\item[(i)] $\displaystyle{\overline{p}(4n) = 4 \sum_{k=0}^{n-2} \sum_{j=1}^{\left\lceil\sqrt{n-k}\right\rceil-1 } \overline{p}(k)\overline{p}(2n -k-2j^2) + 2\sum_{k=0}^{n-1}\overline{p}(k)\overline{p}(2n-k) }$ 
		\item[] $\displaystyle{ \qquad\qquad +2\sum_{j=1}^{\lfloor\sqrt{n}\rfloor} \overline{p}^2(n-j^2) + \overline{p}^2(n) }$;
		\item[(ii)] $\displaystyle{ \overline{p}(4n+1) = 4 \sum_{k=0}^{n-1} \sum_{j=1}^{\left[\sqrt{n-k} \right] } \overline{p}(k)\overline{p}\big(2n -k-2j(j-1)\big) + 2\sum_{j=1}^{\left[\sqrt{n+1} \right]}\overline{p}^2\big(n-j(j-1)\big) }$;
		\item[(iii)] $\displaystyle{ \overline{p}(4n+2) = 4 \sum_{k=0}^{n-1} \sum_{j=1}^{\left\lfloor\sqrt{n-k} \right\rfloor } \overline{p}(k)\overline{p}(2n+1 -k-2j^2) + 2\sum_{k=0}^{n}\overline{p}(k)\overline{p}(2n+1-k) }$;
		\item[(iv)] $\displaystyle{ \overline{p}(4n+3) = 4 \sum_{k=0}^n \sum_{j=1}^{\left[\sqrt{n+1-k} \right] } \overline{p}(k)\overline{p}\big(2n+1 -k-2j(j-1)\big) }$.			
	\end{enumerate}	
\end{corollary}

The expansion of $\overline{p}(n)$ by the linear recurrence relation \eqref{Eq:0} requires exactly $\lfloor \sqrt{n} \rfloor$ distinct terms. By Corollary \ref{C1}, we deduce that the expansion of $\overline{p}(4n)$ or $\overline{p}(4n+2)$ by Theorem \ref{T1} requires exactly
$$n+1+\sum_{k=1}^n \left\lfloor \sqrt{k} \right\rfloor$$
distinct terms, while the expansion of $\overline{p}(4n+1)$ or $\overline{p}(4n+3)$ requires exactly
$$\sum_{k=1}^{n+1} \left[ \sqrt{k} \right]$$
distinct terms. Even though the following inequalities
$$\left\lfloor \sqrt{4n+3} \right\rfloor \leqslant \sum_{k=1}^{n+1} \left[ \sqrt{k} \right] < n+1+\sum_{k=1}^n \left\lfloor \sqrt{k} \right\rfloor$$
holds for any positive integer $n$, we will prove that the formula given by Theorem \ref{T1} is more efficient than the formula given by \ref{Eq:0} for $n>8$.

Computing the value of $\overline{p}(n)$ by formula \eqref{Eq:0} requires the values of $\overline{p}(k)$ for all values of $k$ less than $n$. In this case, we use exactly
$$M_1(n) = \sum_{k=1}^n \left\lfloor \sqrt{k} \right\rfloor$$
of the values of $\overline{p}(k)$, $k<n$, to compute $\overline{p}(n)$.  
Let $M_2(n)$ be the number of values of $\overline{p}(k)$, $k\leqslant n/2$, invoked by Theorem \ref{T1} to compute $\overline{p}(n)$.

\begin{example}
	For $n=11$, we can write:
	\begin{align*}
	&  \overline{p}_0 = 1,\\
	&  \overline{p}_1 = 2\overline{p}_0=2,\\
	&  \overline{p}_2 = 2\overline{p}_1=4,\\
	&  \overline{p}_3 = 2\overline{p}_2=8,\\
	&  \overline{p}_4 = 2(\overline{p}_3 - \overline{p}_0)=14,\\
	&  \overline{p}_5 = 2(\overline{p}_4-\overline{p}_1)=24,\\
	& \overline{p}_{11} = 4\big(\overline{p}_0(\overline{p}_5+\overline{p}_1)+\overline{p}_1\overline{p}_4+\overline{p}_2\overline{p}_3\big)= 344.
	\end{align*}	
	It is clear that that $M_2(11)=14$. On the other hand, we have 
	$$M_1(11)=1+1+1+2+2+2+2+2+3+3+3=22.$$ 
\end{example}	

The following result shows that the sequence $\left\lbrace M_2(n)/M_1(n) \right\rbrace_{n>0} $ is convergent and its limit is less than $1/2$. This fact confirms that the formula given by Theorem \ref{T1} is more efficient than the formula \eqref{Eq:0}.
\allowdisplaybreaks{
	\begin{theorem}\label{T2}
		$$\lim_{n\to\infty} \frac{M_2(n)}{M_1(n)} = \frac{1}{8}+\sqrt{\frac{1}{8}}=0.47855\ldots.$$
	\end{theorem}
	
	We illustrate this theorem in the following four tables.

	\begin{table}[H]
		\caption{Values for $M_1(n)$ and $M_2(n)$ with $n\equiv 1\pmod 4$}
		\label{tab1}
		\centering
		\small\addtolength{\tabcolsep}{-3.2pt}
		\begin{tabular}{ccccccccccc}
			\multicolumn{11}{l}{Values for $M_1(n)$ and $M_2(n)$ with $n\equiv 1\pmod 4$}\\
			\hline \\[-2ex]
			$n$      & 1 & 5 &  9 & 13 & 17 & 21 & 25 & 101 & 1001 & 10001\\ [0.5ex]
			\hline \\[-2ex]
			$M_1(n)$ & 1 & 7 & 16 & 28 & 42 & 58 & 75 & 635 & 20646 & 661850 \\[0.5ex]
			$M_2(n)$ & 2 & 6 & 13 & 20 & 27 & 36 & 47 & 337 & 10149 & 319225 \\[0.5ex]
			$M_2/M_1$ & 2.000 & 0.857 & 0.812 & 0.714 & 0.642 & 0.620 & 0.626 & 0.530 & 0.491 & 0.482 \\[0.5ex]
			\hline \\
				\end{tabular}
			\end{table}
			\begin{table}[h]
				\caption{Values for $M_1(n)$ and $M_2(n)$ with $n\equiv 2\pmod 4$}
				\label{tab2}
				\centering
				\small\addtolength{\tabcolsep}{-3.2pt}
				\begin{tabular}{ccccccccccc}
			\multicolumn{11}{l}{Values for $M_1(n)$ and $M_2(n)$ with $n\equiv 2\pmod 4$}\\
			\hline \\[-2ex]
			$n$      & 2 & 6 & 10 & 14 & 18 & 22 & 26 & 102 & 1002 & 10002\\ [0.5ex]
			\hline \\[-2ex]
			$M_1(n)$ & 2 & 9 & 19 & 31 & 46 & 62 & 80 & 645 & 20677& 661950 \\[0.5ex]
			$M_2(n)$ & 3 & 9 & 17 & 25 & 35 & 46 & 57 & 376 & 10526& 322972\\[0.5ex]
			$M_2/M_1$ & 1.500 & 1.000 & 0.894 & 0.806 & 0.760 & 0.741 & 0.712 & 0.582 & 0.509 & 0.487\\[1ex]
			\hline \\
				\end{tabular}
			\end{table}	
			\begin{table}[h]
				\caption{Values for $M_1(n)$ and $M_2(n)$ with $n\equiv 3\pmod 4$}
				\label{tab3}
				\centering
				\small\addtolength{\tabcolsep}{-3.2pt}
				\begin{tabular}{ccccccccccc}
			\multicolumn{11}{l}{Values for $M_1(n)$ and $M_2(n)$ with $n\equiv 3\pmod 4$}\\
			\hline \\[-2ex]
			$n$      & 3 & 7 & 11 & 15 & 19 & 23 & 27 & 103 & 1003 & 10003 \\ [0.5ex]
			\hline \\[-2ex]
			$M_1(n)$ & 3 & 11 & 22 & 34 & 50 & 66 & 85 & 655 & 20708 & 662050 \\[0.5ex]
			$M_2(n)$ & 3 &  7 & 14 & 21 & 29 & 38 & 48 & 340 & 10156 & 319246\\[0.5ex]
			$M_2/M_1$ & 1.000 & 0.636 & 0.636 & 0.617 & 0.580 & 0.575 & 0.564 & 0.519 & 0.490 & 0.482 \\[1ex]
			\hline \\
				\end{tabular}
			\end{table}	
			\begin{table}[h]
				\caption{Values for $M_1(n)$ and $M_2(n)$ with $n\equiv 0\pmod 4$}
				\label{tab4}
				\centering
				\small\addtolength{\tabcolsep}{-3.2pt}
				\begin{tabular}{ccccccccccc}
			\multicolumn{11}{l}{Values for $M_1(n)$ and $M_2(n)$ with $n\equiv 0\pmod 4$}\\
			\hline \\[-2ex]
			$n$      & 4 & 8 & 12 & 16 & 20 & 24 & 28 & 104 & 1004 & 10004\\ [0.5ex]
			\hline \\[-2ex]
			$M_1(n)$ & 5 & 13 & 25 & 38 & 54 & 70 & 90 & 665 & 20739 & 662150\\[0.5ex]
			$M_2(n)$ & 8 & 15 & 23 & 33 & 44 & 55 & 66 & 395 & 10580 & 323144\\[0.5ex]
			$M_2/M_1$ & 1.600 & 1.153 & 0.920 & 0.868 & 0.814 & 0.785 & 0.733 & 0.593 & 0.510 & 0.488\\[1ex]
			\hline \\
		\end{tabular}
	\end{table}	
}

\section{Proof of Theorem \ref{T1}}

We denote by $\overline{p_o}(n)$ the number of overpartitions of $n$ into odd parts. It is well known that the generating function for $\overline{p_o}(n)$ is given by
\begin{equation}\label{GFOVP2}
\sum_{n=0}^{\infty} \overline{p_o}(n) q^n = \frac{(-q;q^2)_{\infty}}{(q;q^2)_{\infty}}.
\end{equation}
The expression of this generating function firstly appeared in the following series-product identity 
$$\sum_{n=0}^{\infty} \frac{(-1;q)_n q^{n(n+1)/2} }{(q;q)_n} =  \frac{(-q;q^2)_{\infty}}{(q;q^2)_{\infty}}$$
that was published by Lebesgue \cite{Lebesgue} in $1840$. More recently, the generating function \eqref{GFOVP2} for $\overline{p_o}(n)$ appeared in the works of Bessenrodt \cite{Bessenrodt}, Santos and Sills \cite{Santos}.  
Arithmetic properties of the function $\overline{p_o}(n)$ have been investigated later by Hirschhorn and Sellers \cite{Hirschhorn06}. 

In order to prove Theorem \ref{T1}, we present some relationships between $\overline{p}(n)$ and $\overline{p_o}(n)$.

\begin{lemma}\label{L1}
	For $n\geqslant 0$,
	\begin{equation*}
	\overline{p}(n) = \sum_{k=0}^{\lfloor n/2 \rfloor} \overline{p}(k) \overline{p_o}(n-2k).
	\end{equation*}
\end{lemma}

\begin{proof}
	Having
	$$ \frac{(-q;q)_\infty}{(q;q)_\infty} = \frac{(-q^2;q^2)_\infty}{(q^2;q^2)_\infty} \cdot
	\frac{(-q;q^2)_\infty}{(q;q^2)_\infty},$$
	we obtain
	$$ 
	\sum_{n=0}^{\infty} \overline{p}(n) q^n
	= \left( \sum_{n=0}^{\infty} \overline{p}(n) q^{2n} \right) 
	\left( \sum_{n=0}^{\infty} \overline{p_o}(n) q^n \right).
	$$ 
\end{proof}

\begin{lemma}\label{L2}
	For $n\geqslant 0$,
	\begin{enumerate}
		\item[(i)] $\displaystyle{ \overline{p_o}(2n) = \overline{p}(n) +2 \sum_{k=1}^{\left\lfloor \sqrt{n/2} \right\rfloor} \overline{p}(n-2k^2) };$
		\item[(ii)] $\displaystyle{ \overline{p_o}(2n+1) = 2 \sum_{k=0}^{\left\lfloor \sqrt{n/2} \right\rfloor} \overline{p}\big(n-2k(k+1)\big) }.$		
	\end{enumerate}
\end{lemma}

\begin{proof}
	The Jacobi triple product identity \cite[Theorem 11]{Andrews04} can be expressed in terms of the Ramanujan theta function as follows
	$$(-q;qx)_{\infty} (-x;qx)_{\infty} (qx;qx)_{\infty} = \sum_{n=-\infty}^{\infty}q^{n(n+1)/2} x^{n(n-1)/2}, \qquad \left| qx\right|<1.$$
	Replacing $x$ by $q$ in this relation, we obtain
	\begin{eqnarray*}
		(-q;q^2)^2_\infty (q^2;q^2)_\infty = \sum_{n=-\infty}^\infty q^{n^2}
	\end{eqnarray*}
	On the other hand, we have
	\begin{eqnarray*}
		(-q;q^2)^2_\infty (q^2;q^2)_\infty 
		= \frac{(-q;q^2)_\infty (q^2;q^4)_\infty (q^2;q^2)_\infty}{(q;q^2)_\infty} 
		= \frac{(-q;q^2)_\infty }{(q;q^2)_\infty } 	
		\cdot \frac{ (q^2;q^2)_\infty}{(-q^2;q^2)_\infty} 
	\end{eqnarray*}
	Thus we deduce that
	\begin{eqnarray*}
		\frac{(-q;q^2)_\infty }{(q;q^2)_\infty } 
		= \frac {(-q^2;q^2)_\infty} { (q^2;q^2)_\infty} \sum_{n=-\infty}^\infty q^{n^2}
	\end{eqnarray*}
	and
	\begin{eqnarray*}
		\frac{(q;q^2)_\infty }{(-q;q^2)_\infty } 
		= \frac {(-q^2;q^2)_\infty} { (q^2;q^2)_\infty} \sum_{n=-\infty}^\infty (-1)^n q^{n^2}.
	\end{eqnarray*}
	Then we can write the following identities
	$$
	\frac{1}{2} \left( \frac{(-q;q^2)_\infty}{(q;q^2)_\infty} + \frac{(q;q^2)_\infty}{(-q;q^2)_\infty} \right) 
	= \frac{(-q^2;q^2)_\infty}{(q^2;q^2)_\infty} \sum_{n=-\infty}^{\infty} q^{4n^2}
	$$
	and
	$$
	\frac{1}{2} \left( \frac{(-q;q^2)_\infty}{(q;q^2)_\infty} - \frac{(q;q^2)_\infty}{(-q;q^2)_\infty} \right) 
	= \frac{(-q^2;q^2)_\infty}{(q^2;q^2)_\infty} \sum_{n=-\infty}^{\infty} q^{(2n+1)^2}.
	$$
	Considering the generating functions of  $\overline{p}(n)$ and $\overline{p_o}(n)$, we obtain the relations
	$$\sum_{n=0}^\infty  \overline{p_o}(2n) q^{2n} 
	= \left( \sum_{n=0}^\infty  \overline{p}(n) q^{2n}\right) 
	\left( \sum_{n=-\infty}^{\infty} q^{4n^2} \right)
	$$
	and
	$$\sum_{n=0}^\infty  \overline{p_o}(2n+1) q^{2n+1} 
	= \left( \sum_{n=0}^\infty  \overline{p}(n) q^{2n}\right) 
	\left( \sum_{n=-\infty}^{\infty} q^{(2n+1)^2} \right).
	$$
	that can be written as
	$$\sum_{n=0}^\infty  \overline{p_o}(2n) q^{n} 
	= \left( \sum_{n=0}^\infty  \overline{p}(n) q^{n}\right) 
	\left( \sum_{n=-\infty}^{\infty} q^{2n^2} \right)
	$$
	and
	$$\sum_{n=0}^\infty  \overline{p_o}(2n+1) q^{n} 
	= \left( \sum_{n=0}^\infty  \overline{p}(n) q^{n}\right) 
	\left( \sum_{n=-\infty}^{\infty} q^{2n^2+2n} \right).
	$$
	Equating the coefficients of $q^n$ in the last identity gives the following decomposition of $\overline{p_o}(2n)$ in terms of the overpartition function $\overline{p}(n)$:
	$$ \overline{p_o}(2n) = \sum_{k=-\infty}^{\infty} \overline{p}(n-2k^2)$$
	and
	$$ \overline{p_o}(2n+1) = \sum_{k=-\infty}^{\infty} \overline{p}\big(n-2k(k+1)\big).$$
	These conclude the proof.
\end{proof}

The proof of Theorem \ref{T1} follows easily considering Lemmas \ref{L1} and \ref{L2}.

\section{Proof of Theorem \ref{T2}}

First, we prove the case $n\equiv 0 \pmod 4$. According to the linear recurrence relation \eqref{Eq:0} and Corollary \ref{C1}.(i), we have:
\begin{equation*}
M_1(4n) = \sum_{k=1}^{4n} \left\lfloor\sqrt{k}\right\rfloor
\end{equation*}
and
\begin{align*}
M_2(4n) & = M_1(2n) + n-1+\sum_{k=0}^{n-2} \Big(\left\lceil \sqrt{n-k}\right\rceil -1\Big) +2n +2\left\lfloor \sqrt{n} \right\rfloor+2 \\
& = 2n+1 + 2\left\lfloor \sqrt{n} \right\rfloor +\sum_{k=1}^{n} \left\lceil \sqrt{k}\right\rceil+\sum_{k=1}^{2n} \left\lfloor\sqrt{k}\right\rfloor.
\end{align*}
Considering that
$$\sqrt{k}-1 < \left\lfloor \sqrt{k} \right\rfloor \leqslant \sqrt{k}\leqslant \left\lceil \sqrt{k} \right\rceil < \sqrt{k}+1,$$
we can write
\begin{equation*}
\sum_{k=1}^{4n} \sqrt{k} - 4n < M_1(4n) \leqslant \sum_{k=1}^{4n} \sqrt{k}
\end{equation*}
and
\begin{equation*}
2\sqrt{n}-1+\sum_{k=1}^{n} \sqrt{k}+\sum_{k=1}^{2n} \sqrt{k} < M_2(4n) < 3n+1 + 2\sqrt{n}+\sum_{k=1}^{n} \sqrt{k}+\sum_{k=1}^{2n} \sqrt{k}.
\end{equation*}
On the other hand, by Merca \cite[Theorem 1]{Merca}, for $n>0$ we have
\begin{equation*}
\left( \frac{2n}{3}+\frac{1}{8}-\frac{1}{8\sqrt{n+1}}\right) \sqrt{n+1} < \sum_{k=1}^{n}\sqrt{k} < \left(\frac{2n}{3}+\frac{1}{6}-\frac{1}{6\sqrt{n+1}} \right)\sqrt{n+1}.
\end{equation*}
These double inequalities allows us to deduce that
$$\lim_{n\to\infty} \frac{M_2(4n)}{M_1(4n)} = \frac{1}{8} + \sqrt{\frac{1}{8}}.$$
In a similar way, we prove that
$$\lim_{n\to\infty} \frac{M_2(4n+r)}{M_1(4n+r)} = \frac{1}{8} + \sqrt{\frac{1}{8}},$$
for each $r=1,2,3$.

\section{Concluding remarks}

A new algorithm for computing the overpartition function $\overline{p}(n)$ has been introduced in this paper.
Although our algorithm is not the fastest way to compute an isolated value of $\overline{p}(n)$, it works fine for $n$ up to a few million.  In addition, this algorithm allows the computation to be split across multiple processors more easily than an algorithm based power series inversion. It remains an open problem whether there is a fast way to compute the isolated value $\overline{p}(n)$ using purely algebraic methods.


\bigskip

\end{document}